\documentclass[10pt,letterpaper,twoside,web]{ieeecolor}
\usepackage{generic}
\usepackage{cite}
\usepackage{amsmath,amssymb,amsfonts}
\usepackage{algorithmic}
\usepackage{graphicx}
\usepackage{eucal}
\usepackage{textcomp}
\usepackage{gatechamberlab}
\renewcommand{\journalname}{IEEE CONTROL SYSTEMS LETTERS}

\def\BibTeX{{\rm B\kern-.05em{\sc i\kern-.025em b}\kern-.08em
		T\kern-.1667em\lower.7ex\hbox{E}\kern-.125emX}}
\begin{document}
	\title{Input-to-State Safety With Control Barrier Functions}
	\author{Shishir Kolathaya, \IEEEmembership{Member, IEEE}, Aaron D. Ames, \IEEEmembership{Senior, IEEE}
		\thanks{Manuscript received March 6, 2018; revised May 26, 2018; accepted June 18, 2018. Date of publication July 6, 2018; date of current version July 23, 2018. This work was supported in part by the DST INSPIRE Faculty Fellowship under Grant IFA17-ENG212, and in part by NSF under Grant 1724464. Recommended by Senior Editor C. Prieur. (Corresponding author: Shishir Kolathaya.)}
		\thanks{S. Kolathaya is with the Robert Bosch Center for Cyber Physical Systems, Indian Institute of Science, Bengaluru 560012, India (e-mail: shishirk@iisc.ac.in).}
		\thanks{A. D. Ames is with the Department of Mechanical and Civil Engineering, California Institute of Technology, Pasadena, CA 91125 USA (e-mail: ames@caltech.edu).}
		\thanks{Digital Object Identiﬁer 10.1109/LCSYS.2018.2853698}
	}
	
	\maketitle
	\thispagestyle{empty}
	\pagestyle{empty}
	
	\begin{abstract}
This letter presents a new notion of {\it input-to-state safe control barrier functions} (ISSf-CBFs), which ensure safety of nonlinear dynamical systems under input disturbances. Similar to how safety conditions are specified in terms of forward invariance of a set, {\it input-to-state safety} (ISSf) conditions are specified in terms of forward invariance of a slightly larger set. In this context, invariance of the larger set implies that the states stay either inside or very close to the smaller safe set; and this closeness is bounded by the magnitude of the disturbances. The main contribution of the letter is the methodology used for obtaining a valid ISSf-CBF, given a control barrier function (CBF). The associated universal control law will also be provided. Towards the end, we will study unified quadratic programs (QPs) that combine control Lyapunov functions (CLFs) and ISSf-CBFs in order to obtain a single control law that ensures both safety and stability in systems with input disturbances.
	\end{abstract}
	

\section{Introduction}
\label{sec:introduction}

Real-time safety in dynamical systems has been receiving a lot of attention of late: \cite{7040372,7782377,7524935,7039737,XU201554}. Safety was initially studied in 2005, when {\it barrier certificates} were introduced in \cite{PRAJNA2005526} that certified whether a given dynamical system was safe or not. 
This was later adapted for real-time safety critical control via
{\it control barrier functions} (CBFs), which were first introduced in \cite{WIELAND2007462}. Yet these CBFs did not allow for safety to be imposed on top of an existing controller or in conjunction with stability conditions; both of which are necessary in robotic systems.

Real-time optimization based controllers can be implemented on robotic systems like quadrotors, automotive systems, and mobile robots due to the accessibility of high processing capability in remarkably small dimensions. With this access to technology, there were several key contributions in realizing a unifying controller that ensures both safety and stability via quadratic programs (QPs) \cite{7040372,7524935,XU201554}. In particular \cite{7040372,XU201554} developed a new notion of control barrier functions together with conditions that are necessary and sufficient for set invariance. Therefore, CLFs and CBFs can be encoded as constraints in a single QP that can either a) ensure both stability and safety or b) prioritize safety or stability over the other depending upon the applications. 

To date, this new notion of CBFs have been successfully implemented in automotive systems \cite{xu2017realizing}, flying systems \cite{wang2017safe}, multi-robot systems \cite{wang2017safety}, and also walking robots \cite{RSS2017_DiscreteTerrain_Walking}. It has also been observed that in all these systems uncertainties were a common occurrence, and we had no means to characterize them in a formal manner.
%
For example, in \cite{xu2017realizing}, safety was considered in the context of lane keeping, and a barrier function was used to encode lane boundaries. When the lanes were narrow, small errors in sensing lead to breaching of the lane limits. 
In \cite{wang2017safe}, safety was imposed between the agents in the form of radial distances, and delay in actuation and sensing resulted in collisions. A standard workaround to address these types of uncertainties is to allow some buffer in the margins, but, there is no existing literature that allows to make an estimate of this buffer while providing a means to study safety under uncertainties. This is contrary to the fact that there are existing notions of robustness for modeling and sensing based uncertainties in the field of stability \cite{sontag1995characterizations}.

Safety and stability have very similar properties and the construction of Lyapunov-like conditions for {\it barrier functions} (BFs) enabled the translation of concepts from the field of stability analysis to the domain of safety and characterizations thereof. There are key contributions in converse Lyapunov-like theorems \cite{7782377}, construction of barrier functions via sum of squares \cite{doi:10.1137/050645178}, and especially, robustness analysis via the notion of {\it input-to-state safety} (ISSf) \cite{7799254}. {\it Input-to-state safety} (ISSf) is the equivalent of {\it input-to-state stability} (ISS) \cite{sontag1995characterizations}, which is an elegant theory used to characterize stability of nonlinear systems under input disturbances. 

The main objective of this letter is to build upon the notion of ISSf presented in \cite{7799254}, extending it in the context of Lypunov-like characterizations of ISSf. Therefore our focus will be on the construction of \textbf{input-to-state safe (or safeguarding) control barrier functions} (ISSf-CBFs), which are crucial for robust implementations of real-time safety critical controllers in nonlinear systems. We will study CBFs, and the associated ISSf-CBFs, and also realize a unified quadratic program (QP) based formulation that ensures both safety and stability in nonlinear systems under input disturbances.






	We will first formally define the notion of \textbf{ input-to-state safety} (ISSf) w.r.t. sets. ISSf w.r.t. systems was originally defined in \cite{romdlony2017robustness}. Our choice for an alternative definition is motivated by the problem definition in \cite{7040372}. 
	Having defined ISSf, we will also define \textbf{ input-to-state safe control barrier functions} (ISSf-CBFs). Similar to how CBFs are constructed for ensuring {\it safety} of sets, we will construct ISSf-CBFs for ensuring ISSf of sets. We will establish that given a CBF, an associated ISSf-CBF can be constructed that always ensures that the states stay either inside or very close to the safe set.
	We will finally construct a \textbf{quadratic program} (QP) that contains both CLF and ISSf-CBF based constraints that results in a unified safeguarding-stabilizing controller under input disturbances. This will be further demonstrated in two examples.

A preliminary on CBFs will be provided in Section \ref{sec:cbf}. ISSf will be described in Section \ref{sec:issf}, ISSf-CBFs will be described in Section \ref{sec:issfcbf}, and finally, the unification of stability and input-to-state safety via QPs will be described in Section \ref{sec:unification}.
\section{Preliminary on control barrier functions}\label{sec:cbf}
In this section, we will study {\it barrier functions} and also state their relationships with forward invariance of a set (see \cite{7782377}, wherein the version we considered was called {\it zeroing barrier functions}). We consider a system of the form:
\begin{align}
\label{eq:system}
\dot x = f(x),
\end{align}
where $x\in \R^n$, $f: \R^n \to \R^n$ is locally Lipschitz. Given an initial condition $x_0 := x(t_0) \in \R^n$, there exists a maximum time interval $I(x_0) = [t_0, t_{\max})$ such that $x(t)$ is the unique solution to \eqref{eq:system} on $I(x_0)$; in the case when \eqref{eq:system} is forward complete, $t_{\max} = \infty$. A set $\mathcal{S}\subset \R^n$ is forward invariant w.r.t. \eqref{eq:system} if for every $x_0 \in \mathcal{S}$, $x(t) \in \mathcal{S}$ for all $t\in I(x_0)$. If $\mathcal{S}$ is forward invariant, then we call the set $\mathcal{S}$ {\it safe}.

Given a closed set $\mathcal{C}\subset \R^n$ (which is a strict subset of $\R^n$), we determine conditions such that it is forward invariant. $\mathcal{C}$ is defined as
\begin{align}
\label{eq:cset1}
	\mathcal{C} &= \{ x\in \R^n : h(x) \geq 0\},  \\
\label{eq:cset2}
   \partial	\mathcal{C} &= \{ x\in \R^n : h(x) = 0\}, \\
\label{eq:cset3}
	\mathrm{Int}(\mathcal{C}) &= \{ x\in \R^n : h(x) > 0\}, 
\end{align}
where $h:\R^n \to \R$ is a continuously differentiable function. It is also assumed that $\mathrm{Int}(\mathcal{C})$ is non-empty and $\mathcal{C}$ has no isolated points, i.e.,
	$\mathrm{Int}(\mathcal{C}) \neq \emptyset$, and $\overline{\mathrm{Int}(\mathcal{C})} = \mathcal{C}$.

\newsec{Notation.}\label{subsec:notation}
%
A continuous function $\alpha: [0,a) \to [0,\infty)$ for some $a>0$ is said to belong to {\it class} $\classK$ if it is strictly increasing and $\alpha(0)=0$. Here, $a$ is allowed to be $+\infty$.
%
A continuous function $\alpha : [0,\infty) \to [0,\infty)$ is said to belong to {\it class } $\classK_\infty$ if it is strictly increasing, $\alpha(0)=0$, and $\alpha(r) \to \infty$ as $r\to\infty$.
%
%
A continuous function $\alpha: (-b,c) \to (-\infty,\infty)$ is said to belong to extended class $\classK$ for some $b>0$, $c>0$ if it is strictly increasing and $\alpha(0)=0$ (see \cite[Definition 1]{XU201554}). Here again, $b$, $c$ are allowed to be $+\infty$.
To indicate the domains, we will denote class $\classK$ and extended class $\classK$ functions as $\classK_{[0,a)}$, $\classK_{(-b,c)}$ respectively.

Given the state $x$, we denote its Euclidean norm as $|x|$. 
For a signal $d:\R_{\geq 0} \to \R^m$, 
its $\mathbb{L}^m_\infty$ norm is given by $\|d\|_\infty : = \esssup_t |d(t)|$. 

\subsection{Barrier functions}
\label{subsection:zbf}
Given the set $\mathcal{C}$, our objective is to establish safety by taking into consideration a larger set $\mathcal{D} \subseteq \R^n$. 
This is similar in analogy to establishing local stability results for systems (see Remark \ref{rm:whyd}). By assuming that $\mathcal{D}$ is open, we have the following definition of a {\it barrier function} (BF). 
\begin{definition}\label{def:zbf}{\it
	For the dynamical system \eqref{eq:system}, a continuously differentiable function $h:\R^n \to \R$ is a {\bf barrier function} (BF) for the set $\mathcal{C}\subset\R^n$ defined by \eqref{eq:cset1}-\eqref{eq:cset3}, if there is an open set $\mathcal{D}$ with $\mathcal{C}\subset \mathcal{D} \subseteq \R^n$, an $\alpha\in\classK_{(-b,c)}$ with $b,c$ appropriately chosen, such that for all $x\in\mathcal{D}$, 
	\begin{align}
		L_f h(x) \geq - \alpha(h(x)).
	\end{align}}
\end{definition}
\noindent Here $L_f h$ is the Lie derivative of $h$ w.r.t. $f$. $b,c$ must be picked such that $h(x) \in (-b,c)$. 
See Remark \ref{rm:whyd}. 

\begin{remark}\label{rm:whyd}
	To illustrate the importance of $\mathcal{D}$ we consider the following differentiable function:
	\begin{align}
		h(x) = \left \{  \begin{array}{ccc}
		-1 & {\rm{if}} & x < - 1 \\
		\sin\left(\frac{\pi}{2}x\right) & {\rm{if}} & - 1 \leq x <  1 \\
		1 & {\rm{if}} & 1 \leq x 		
		\end{array}     \right . .
	\end{align}
	It can be verified that $\mathcal{C} = \{ x: x \geq 0 \}$. In addition, if $ x < - 1$, then $\frac{\partial  h}{\partial x} =0$. In other words, $h$ cannot be a valid BF for any $x\in (-\infty,-1)$. In order to {\it not} restrict our choices of $h$, we typically pick a smaller set $\mathcal{D}$ that contains $\mathcal{C}$ for $x$. For the example above, $\mathcal{D}=(-1,\infty)$.

	With this viewpoint, we will make a few of the notations precise (which will be useful for defining the comparison functions later on): 
	\begin{align}
	 \label{eq:bce}
	 b : = - \inf_{x\in \R^n} h(x), \quad
	 c : = \sup_{x\in \R^n} h(x) , \quad
	 e : = -\lim_{r\to -b} \alpha(r).
	\end{align}
Note that, here, $-b,c$ are the boundaries of the domain of the extended class $\classK$ function $\alpha$ introduced in Definition \ref{def:zbf}. This ensures that $\alpha(h(x))$ is well defined for all $x$. We will define $\mathcal{D}$ for the rest of the letter as
 	\begin{align}\label{eq:dset}
 		\mathcal{D} := \{ x\in \R^n : h(x) + b >0 \}.
 	\end{align}
\end{remark}

\subsection{Control barrier functions}\label{subsection:cbf}
Having defined the BF, $h$, we can now define {\it control barrier functions} (CBFs)\footnote{which were called {\it zeroing control barrier functions} in \cite{7782377}.}. Consider the affine control system:
\begin{align}
\label{eq:affinesystem}
	\dot x = f(x) + g(x) u,
\end{align}
with $f:\R^n \to \R^n$, $g:\R^n \to \R^{n\times m}$ being locally Lipschitz, $x\in \R^n$, and $u\in\ControlInput\subset\R^m$. When the set $\mathcal{C}$ is not forward invariant under the natural dynamics of the system, $\dot x = f(x)$, we are interested in the controller $k:\R^n \to \R^m$, that can be specified that will ensure invariance of $\mathcal{C}$. We call this controller a {\it safeguarding} controller w.r.t. the set $\mathcal{C}$. We can obtain a suitable safeguarding controller via CBFs.

\begin{definition}\label{def:zcbf}{\it
	Given a set $\mathcal{C}\subset \R^n$ defined by \eqref{eq:cset1}-\eqref{eq:cset3} for a continuously differentiable function $h:\R^n \to \R$, the function $h$ is called a {\bf control barrier function} (CBF) defined on the open set $\mathcal{D}$ \eqref{eq:dset} with $\mathcal{C}\subset\mathcal{D}\subseteq\R^n$, if there exists a set of controls $\mathbb{U}$, and an $\alpha\in\classK_{(-b,c)}$ such that for all $x\in \mathcal{D}$,
	\begin{align}\label{eq:zcbfeq}
		\sup_{u \in \ControlInput } \left [ L_f h(x) + L_g h(x) u \right ] \geq - \alpha(h(x)).
	\end{align}
}
\end{definition}
\noindent Here $L_fh$, $L_gh$ are the Lie derivatives. If $\mathcal{C}$ is compact, the BF $h$ not only ensures forward invariance, but also ensures asymptotic stability of $\mathcal{C}$ (see \cite[Proposition 4]{XU201554}). 

\section{Preliminary on input-to-state safety}\label{sec:issf}
The notion of {\it input-to-state safety} (ISSf) was first defined in \cite{7799254} and then in \cite{romdlony2017robustness}, wherein the problem formulation was slightly different, i.e., a set of unsafe states $\mathcal{D}_u\subset\R^n$ was defined, and the goal was to stay {\it away} from this unsafe set. 
%
In this manuscript, since our goal is to stay in the superlevel set $\mathcal{C}$ (see \cite[1-2]{XU201554} for the advantages), we will redefine the notion of ISSf for the problem definition and notations provided here (and also in \cite{7782377}). In addition, in formulations similar to the notion of {\it input-to-state stability} \cite{sontag1989smooth}, we consider a safeguarding controller  $k:\R^n \to \R^m$ and posit that additional disturbance $d$ is added to the safeguarding controller. This is similar to the construction of {\it input-to-state stabilizing} controllers in \cite{sontag1989smooth}. We intended to apply a safeguarding controller $k(x)$, but instead $k(x)+d(t)$ was applied to the actual control system \eqref{eq:affinesystem}. Accordingly, we have the following dynamical system: 
\begin{align}\label{eq:affinesystemredefined}
	\dot x = \bar f(x) + g(x) d(t), \:\: {\rm{where}} \:\: \bar f(x) := f(x) + g(x) k(x).
\end{align}
We make preliminary assumptions that $d\in \mathbb{L}_\infty^m$. 
Given this problem setup, the goal 
is to ensure that the states remain either in the superlevel set $\mathcal{C}$, or at least close to $\mathcal{C}$. The {\it closeness} to the superlevel set is directly related to the {\it smallness} of the disturbance input $d$.

We say that the set $\mathcal{C}$ is {\it safe} if it is forward invariant. Accordingly, we say that $\mathcal{C}$ is {\it input-to-state safe} (ISSf) if a slightly larger set $\mathcal{C}_d \supseteq \mathcal{C}$ is forward invariant. Similar to \eqref{eq:cset1}-\eqref{eq:cset3}, we define the set $\mathcal{C}_d$ as
 \begin{align}\label{eq:ce}
\mathcal{C}_d = \{ x\in\R^n : h(x) + \gamma (\|d\|_\infty) \geq 0 \}, \\
\label{eq:cdset2}
\partial \mathcal{C}_d = \{ x\in\R^n : h(x) + \gamma (\|d\|_\infty) = 0  \}, \\
\label{eq:cdset3}
\mathrm{Int}(\mathcal{C}_d) = \{ x\in\R^n : h(x) + \gamma (\|d\|_\infty) > 0 \},
\end{align}
for some $\gamma \in \classK_{[0,a)}$, and $\|d\|_\infty\leq \bar d \in [0,a)$. The constant $a$ satisfies $\lim_{r\to a} \gamma(r)=b$, which ensures that $\gamma(\bar d) < b$, which, in turn, ensures that $\mathcal{C}_d \subset \mathcal{D}$. It is also assumed that $\overline{\mathrm{Int}(\mathcal{C}_d)} = \mathcal{C}_d$. Having defined $\mathcal{C}_d$, we have the following formal definition for the set $\mathcal{C}$ being (locally) ISSf:
\begin{definition}\label{def:ce}
 {\it Given a set $\mathcal{C}\subset \R^n$ defined by \eqref{eq:cset1}-\eqref{eq:cset3} for a continuously differentiable function $h:\R^n \to \R$, and an open set $\mathcal{D}$ \eqref{eq:dset} with $\mathcal{C}\subset \mathcal{D}\subseteq \R^n$, the set $\mathcal{C}$ is called a {\bf (local) input-to-state safe} set, if there exists a $\gamma\in\classK_{[0,a)}$ satisfying $\lim_{r\to a} \gamma(r)=b$, and a constant $\bar d \in [0,a)$, such that for all $d$ satisfying $\|d\|_\infty \leq \bar d$, the set 
 $\mathcal{C}_d\subset \mathcal{D}$ defined by \eqref{eq:ce}-\eqref{eq:cdset3} 
is forward invariant. In other words, the set $\mathcal{C}$ is called an ISSf set if the set $\mathcal{C}_d$, which depends on $d$, is safe.}
\end{definition}
\begin{figure}[t!]\centering
	\includegraphics[width=0.9\columnwidth]{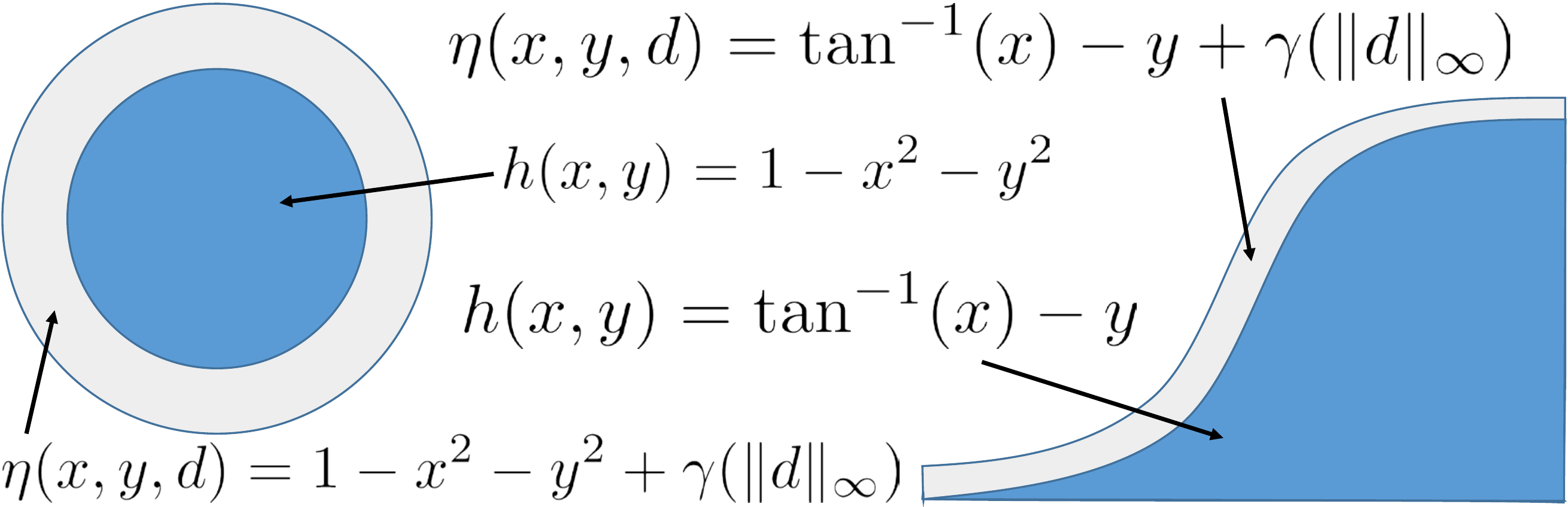}
	\caption{Figure showing some examples of safe and the corresponding  ISSf sets. Blue regions are $\mathcal{C}$, and grey+blue regions are $\mathcal{C}_d$.
	}
	\label{fig:examples}
\end{figure}
\begin{remark}\label{rm:globalissf}
	The above definition of ISSf is w.r.t. sets, and not w.r.t. systems (quite unlike the definitions from \cite{romdlony2017robustness}). See \figref{fig:examples} for some examples of ISSf sets. 
If $\|d\|_\infty \geq a$ with $a$ finite, then $\mathcal{C}_d$ may not necessarily be contained in $\mathcal{D}$, and Definition \ref{def:ce} is no longer valid. This is the reason for including $\bar d$, which is less than $a$. 
This can be extended for any arbitrary $\|d\|_\infty$ if the constants $a,b$ are $+\infty$. In other words, the definition is global if $\gamma \in \classK_{[0,\infty)}$ and $\inf_{x\in \R^n} h(x) = -\infty$.

The global notion of ISSf is less interesting due to the fact that large disturbances imply $\mathcal{C}_d$ is large, which, in turn, implies that the states are far into the unsafe zone. Therefore, we will omit the term {\it local} in the ensuing definitions and results for ISSf.
\end{remark}

To motivate the importance of ISSf, we will begin by studying a concrete example. 
\begin{example}\label{example:1}
	Consider the system
	\begin{align}
	\dot x = - x + x^2 u,
	\end{align}
	along with the safe set $\mathcal{C} = \{ x\in \R : h(x) = 2 - x \geq 0\}$. The goal is to ensure that $x(t) \leq 2$ for all $t$. It can be verified that $k(x)\equiv 0$ is, indeed, a safeguarding controller. We have that
	\begin{align}
	\dot h(x) = L_f h(x)  =  x \geq x - 2 =  - h(x),
	\end{align}
	which implies that $h(x(t)) \geq 0$, if $h(x(0)) \geq 0$. Even if $x$ starts from an unsafe zone, it can be verified that $x$ eventually enters the safe set $\mathcal{C}$.
	On the other hand, if a disturbance $d$ is added (i.e., $u=k(x)+d(t)$), we have the following:
	\begin{align}
	\dot h(x,d)  = x - x^2 d(t).
	\end{align}
	If $x_0 = x(0) = 2$ and $d(t)=1$, it can be verified that the state propagates in an unbounded fashion in the unsafe zone.
	\begin{figure}\centering
		\includegraphics[width=0.48\columnwidth]{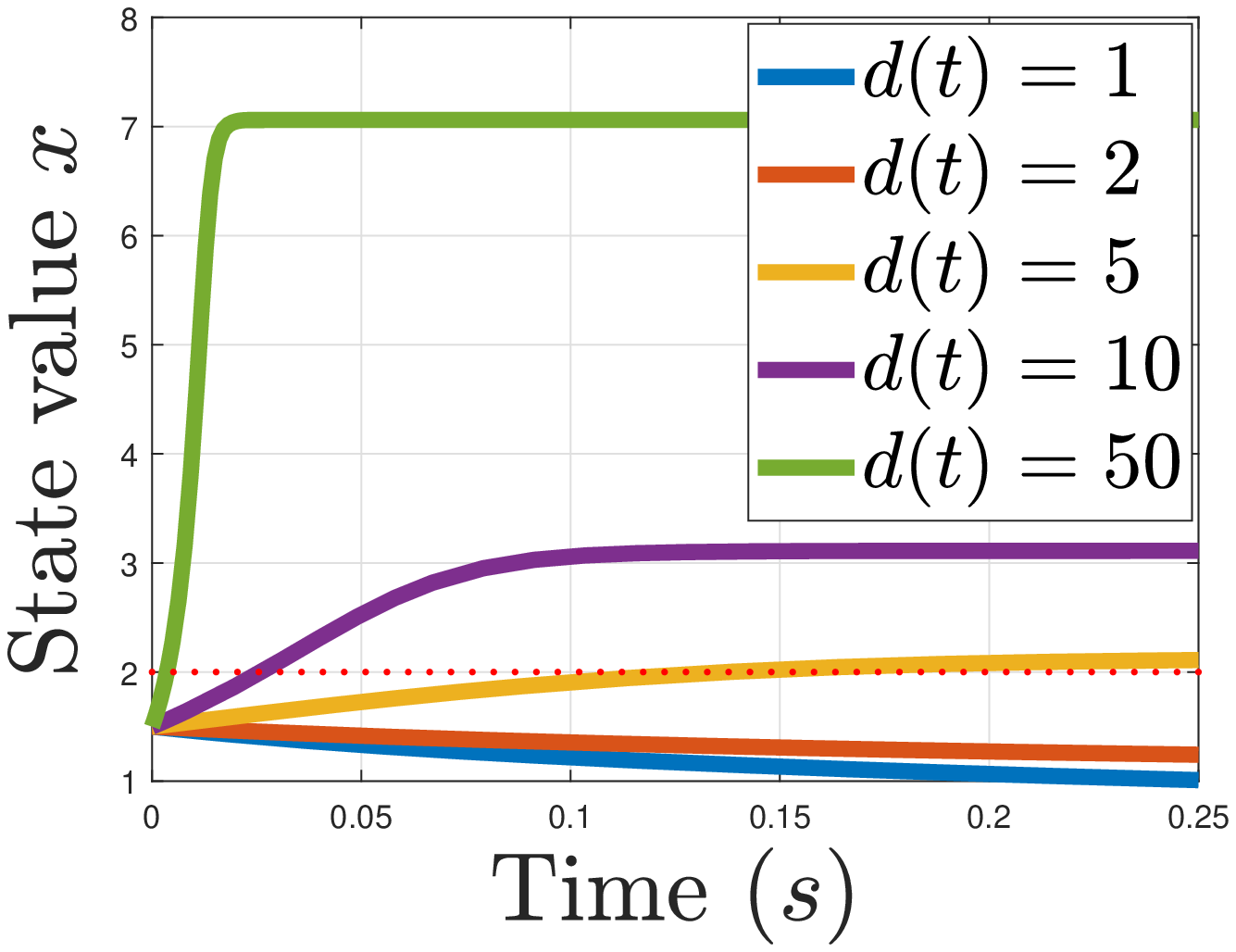}
		\includegraphics[width=0.48\columnwidth]{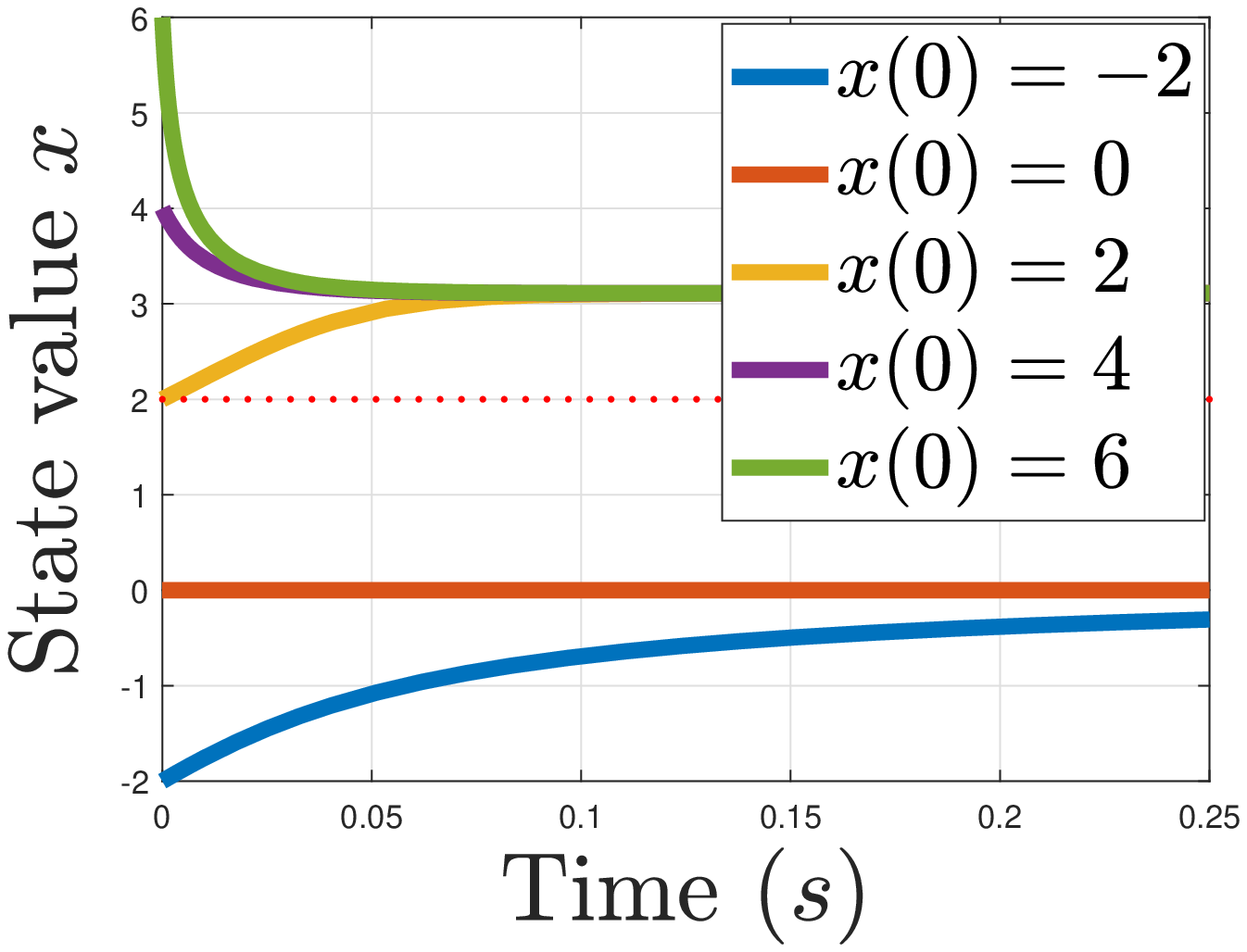}
		\caption{Figure showing the response for different values of input disturbances (left) and initial conditions (right) for a safeguarding controller of the type \eqref{eq:newsafeguarding}. The dashed line corresponds to $x=2$, the boundary of the safe set. 
		}
		\label{fig:example1}
	\end{figure}
	Despite the application of a safeguarding controller, addition of a small disturbance input can drive the states away from the safe set $\mathcal{C}$. In order to address this problem, we propose the following safeguarding controller:
	\begin{align}\label{eq:newsafeguarding}
	k(x) = L_g h(x) = -x^2,
	\end{align}
	which yields
	\begin{align}\label{eq:example1equations}
	\dot h(x,d) &= x + x^4 - x^2 d(t)  \nonumber \\
	&=  x + x^4 - x^2 d(t) + \frac{1}{4} d(t)^2 -\frac{1}{4} d(t)^2 \nonumber \\
	 & \geq x - \frac{1}{4}\|d\|^2_\infty,
	\end{align}
	where the disturbance is replaced with its norm. It can be verified that the states will either stay close to or enter the safe zone for small values of $d$ and as $\|d\|_\infty \to 0$ the state $x$ eventually enters the safe zone (see \figref{fig:example1}). This is the type of formulation we are interested in, and we will use this as the motivation to construct {\it input-to-state safe barrier functions} (ISSf-BFs).
\end{example}
\begin{remark}
If a controller is applied such that the resulting set $\mathcal{C}_d$ is rendered safe, we call this controller an {\it input-to-state safeguarding} controller. It can be observed that the new controller \eqref{eq:newsafeguarding} is, in fact, an {\it input-to-state safeguarding} controller, which will be the basis for the main result here.
\end{remark}

\subsection{Input-to-state safe barrier function}
Having defined the notion of ISSf, we have the following definition of {\it input-to-state safe barrier function} (ISSf-BF). 
\begin{definition}\label{def:issfbf}{\it
Given the dynamical system \eqref{eq:affinesystemredefined}, a continuously differentiable function $h:\R^n \to \R$ is an {\bf input-to-state safe barrier function} (ISSf-BF) for the set $\mathcal{C}\subset \R^n$ defined by \eqref{eq:cset1}-\eqref{eq:cset3}, if there exists an open set $\mathcal{D}$ \eqref{eq:dset} with $\mathcal{C}\subset \mathcal{D} \subseteq \R^n$, an $\alpha\in\classK_{(-b,c)}$, an $\iota\in\classK_{[0,a)}$ satisfying $\lim_{r\to a} \iota(r)=e$, and a constant $\bar d \in [0,a)$, such that $\forall$ $x\in\mathcal{D}$, $\forall$ $\mu\in\R^m$ satisfying $|\mu| \leq \bar d$,
\begin{align}\label{eq:issfzbf}
L_{\bar f} h(x) + L_{g} h(x) \mu \geq - \alpha(h(x)) - \iota(|\mu|).
\end{align}}
\end{definition}
Here $L_{\bar f} h$ is the Lie derivative of $h$ w.r.t. $\bar f$. We have the following result:
\begin{theorem}\label{thm:ceifhisissf}{\it
Given the dynamical system \eqref{eq:affinesystemredefined}, a set $\mathcal{C}\subset\R^n$ defined by \eqref{eq:cset1}-\eqref{eq:cset3} for some continuously differentiable function $h:\R^n \to \R$, and the set $\mathcal{C}_d$ defined by \eqref{eq:ce}-\eqref{eq:cdset3} for some $\gamma\in\classK_{[0,a)}$ satisfying $\lim_{r\to a}\gamma(a)=b$, and $\bar d\in[0,a)$, if $h$ is an ISSf-BF defined on the open set $\mathcal{D}$ \eqref{eq:dset} with $\mathcal{C}\subset \mathcal{D} \subseteq \R^n$, then the set $\mathcal{C}$ is ISSf.}
\end{theorem}
\begin{proof}
	We need to prove that the set $\mathcal{C}_d$ defined by \eqref{eq:ce}-\eqref{eq:cdset3}
	for some $\gamma\in\classK_{[0,a)}$, $\bar d\in[0,a)$, and for all $d$ satisfying $\|d\|_\infty\leq \bar d$, is forward invariant. 
	Given the disturbance $d$, define the new function
	\begin{align}\label{eq:etafunction}
	\eta(x,d) : = h(x) + \gamma(\|d\|_\infty).
	\end{align}
	Since $h$ is an ISSf-BF, we have the following from \eqref{eq:issfzbf}:
	\begin{align}
	\label{eq:etaderivative}
	\dot \eta(x,d)  = \dot h(x,d) 
	& \geq - \alpha(h(x)) - \iota(\|d\|_\infty)  \\
	& = - \alpha  (\eta(x,d) -  \gamma(\|d\|_\infty)) - \iota(\|d\|_\infty), \nonumber
	\end{align}
	where $\eta$ is substituted for $h$. The next steps are similar to proof of \cite[Proposition 1]{7782377}, where we consider the set $\partial \mathcal{C}_d$ \eqref{eq:cdset2}.
	If $x\in \partial \mathcal{C}_d$, then $\eta=0$, and \eqref{eq:etaderivative} reduces to
	\begin{align}\label{eq:gammavalue}
	\dot \eta(x,d) \geq - \alpha  (-  \gamma(\|d\|_\infty)) - \iota(\|d\|_\infty).
	\end{align}
	Substitute for $\beta(r):=- \alpha  (-  r)$, which is a valid class $\classK$ function\footnote{$\beta(0)=0$. If $r_1>r_2>0$, then $-\alpha(-r_1)>-\alpha(-r_2)>0$. The domain of $\beta$ is $[0,b)$ and range is $[0,e)$. Therefore, for $\beta^{-1}$, the domain and range are flipped, which implies that $\beta^{-1} \circ \iota$ is well defined.} only if $r<b$. Therefore, we will pick $\gamma = \beta^{-1} \circ \iota$, and a small enough $\bar d$ such that the following is satisfied
	\begin{align}\label{eq:bardbinequality}
	\beta^{-1} \circ \iota(\bar d) < b.
	\end{align}
	Rest of the proof follows \cite[Proposition 1]{7782377}, i.e., $\dot \eta\geq0$ for $\eta=0$ $\Rightarrow$ $\mathcal{C}_d$ is invariant.
\end{proof}
\begin{remark}\label{rm:issfbfremark}
	The above result can also be applied for {\it exponential}-type barrier functions with $\alpha(h(x)) := \lambda h(x)$ for some $\lambda >0$. 
	This can be substituted in \eqref{eq:etaderivative} to obtain the following invariant set:
	\begin{align}
	\mathcal{C}_d =  \{ x\in\R^n: h(x) + \frac{1}{\lambda} \iota (\|d\|_\infty) \geq 0 \}.
	\end{align}
\end{remark}
It can be verified that for Example \ref{example:1}, $\mathcal{C}_d = \{x : 2 -x + \frac{\|d\|_\infty^2}{4} \geq 0\}$. We will now study ISSf-CBFs that guarantee ISSf of $\mathcal{C}$.

\section{Input-to-state safe control barrier functions}\label{sec:issfcbf}

We will first provide a formal definition for 
{\it input-to-state safe control barrier function} (ISSf-CBF).


\begin{definition}\label{def:issfcbf}{\it
	Given a set $\mathcal{C}\subset \R^n$ defined by \eqref{eq:cset1}-\eqref{eq:cset3} for a continuously differentiable function $h:\R^n \to \R$, the function $h$ is called an {\bf input-to-state safe control barrier function} (ISSf-CBF) defined on the open set $\mathcal{D}$ \eqref{eq:dset} with $\mathcal{C}\subset\mathcal{D}\subseteq\R^n$, if there exists a set of controls $\ControlInput$, an $\alpha\in\classK_{(-b,c)}$, an $\iota\in\classK_{[0,a)}$ satisfying $\lim_{r\to a} \iota(r)=e$, and a constant $\bar d \in [0,a)$, such that $\forall$ $ x\in \mathcal{D}$, $\forall$ $\mu\in\R^m$ satisfying $|\mu| \leq \bar d$,
	\begin{align}\label{eq:issfzcbf}
	\sup_{u \in \ControlInput } \left [ L_f h(x) + L_g h(x) (u+\mu)  \right ] \geq - \alpha(h(x)) - \iota(|\mu|).
	\end{align}
}
\end{definition}


Motivated by constructions developed by Sontag, specifically \cite[equations $(23)$ and $(32)$]{sontag1989smooth}, we can construct ISSf-CBFs in the following manner. 
Given a safeguarding controller $k(x)$, we consider the following controller, which we claim to render the set $\mathcal{C}$ ISSf:
\begin{align}\label{eq:clfcontroller2}
 u(x)  = k(x) +  L_g h(x)^T,
\end{align}
which, incidentally, was also utilized in Example \ref{example:1}. Based on this controller, we have the following theorem which defines a {\it new} ISSf-CBF that renders $\mathcal{C}$ ISSf.
\begin{theorem}\label{lm:zcbflemma1}{\it
Given a set $\mathcal{C}\subset \R^n$ defined by \eqref{eq:cset1}-\eqref{eq:cset3} for a continuously differentiable function $h:\R^n \to \R$, an open set $\mathcal{D}$ \eqref{eq:dset} with $\mathcal{C}\subset \mathcal{D} \subseteq \R^n$, and a set of controls $\ControlInput$, if $h$ satisfies
  \begin{align} \label{eq:newzcbf}
  &  \sup_{u \in \ControlInput }  [ L_f h(x) + L_g h(x) u - L_g h(x) L_g h(x)^T ] \geq - \alpha(h(x)),
 \end{align}
 for some $\alpha\in\classK_{(-b,c)}$, and for all $x\in\mathcal{D}$,
 then $h$ is an ISSf-CBF defined on the set $\mathcal{D}$.}
\end{theorem}
\begin{proof}
 After substituting \eqref{eq:newzcbf} into the derivative of $h$:
 \begin{align}
  \dot h(x,d) &= \sup_{u \in \ControlInput }  [ L_f h(x) + L_g h(x) (u + d(t)) ] \\
		&\geq -\alpha(h(x)) + L_g h(x) L_g h(x)^T + L_g h(x) d(t) \nonumber \\
 &\geq -\alpha(h(x)) + |L_g h(x) |^2 - |L_g h(x)| \|d\|_\infty, \nonumber
\end{align}
since $L_g h L_g h^T = |L_g h|^2$. Adding and subtracting $\frac{1}{4}\|d\|^2_\infty$ yields
\begin{align}\label{eq:issinequality1}
\dot h(x,d)	&\geq -\alpha(h(x)) + \left( |L_g h(x) | - \frac{\|d\|_\infty}{2}\right)^2 -  \frac{\|d\|^2_\infty}{4}  \nonumber \\
		&\geq -\alpha(h(x)) - \frac{\|d\|^2_\infty}{4} ,
\end{align}
which is of the form \eqref{eq:issfzcbf}. 
\end{proof}
\begin{remark}\label{rm:univeralformula}
 Similar to the universal stabilization formula by Sontag \cite{sontag1989universal}, we provide here the universal input-to-state safeguarding formula by using \eqref{eq:clfcontroller2} and Theorem \ref{lm:zcbflemma1}:
 \begin{small}
\begin{align}
	 &u(x) = \left \{ \begin{array}{ll}
		 		0 & \!\! {\rm{if}}  \: B(x) =0 \\
	 			\frac{\left (  - A(x) + \sqrt{A(x)^2 + |B(x)|^4} \right )}{B(x)^T B(x)} B(x) + B(x)  & {\rm{otherwise}}
 				\end{array}   \!\! \right. \nonumber  
\end{align}
\end{small}
Here $A(x) = L_f h(x) + \alpha(h(x))$, $B(x) =  L_g h(x)^T$. It can be verified that $A(x) \geq 0$ whenever $B(x)=0$\footnote{On a different note, if $B(x)\neq 0$ always, then we can use a min-norm controller (like in \cite[(22)-(24)]{7040372}), which is always preferred due to its optimality.}.
\end{remark}
\section{Control Lyapunov Functions and Control Barrier Functions}\label{sec:unification}

In this section, we will study the union of stability and safety i.e., the union of stabilization via control Lyapunov functions (CLFs) and safeguarding via control barrier functions (CBFs). 
These results were well established by \cite{7039737}, and by \cite{7040372}, and the goal here is to extend them for inputs with disturbances.


Consider the system of the form \eqref{eq:affinesystem}, and the corresponding CBF (Definition \ref{def:zcbf}). We know that the set $\mathcal{C}$ defined by \eqref{eq:cset1}-\eqref{eq:cset3} is safe when a safeguarding controller is applied. If in addition, we have a  stabilization problem, we utilize control Lyapunov functions (CLFs). 
\begin{definition}{\it
 \label{def:aclfdefinition}
 A  continuously differentiable function $V :\R^{n} \to \R_{\geq 0}$ is a {\it control Lyapunov function} (CLF), if there exists a set of controls $\ControlInput\subset\R^m$, and class $\classK_\infty$ functions $\underline \alpha,\bar \alpha, \alpha_v$, such that for all $x\in \mathcal{D}$ with $\mathcal{C}\subset\mathcal{D}\subseteq \R^n$,
 \vspace{-2mm}
  \begin{align}\label{eq:aclfdefinition}
  & \underline\alpha( |x|) \leq V(x) \leq  \bar \alpha( |x|) \nonumber \\
  & \inf_{u \in \ControlInput }  [ L_f V(x) + L_g V(x) u ] \leq - \alpha_v(|x|).
 \end{align}}
\end{definition}
\noindent Here $L_fV$ and $L_gV$ are the Lie derivatives. 
Given a CLF $V$ and a BF $h$, they can be combined into a single controller through the use of a quadratic program (QP) in the following manner \cite[Section III.B]{7040372}.

\begin{small}
\vspace{-2mm}
\begin{align}
\label{eq:CLFCBFQP}
	u^*(x) = & \underset{{\mathbf{u} = (u,\delta)\in \R^{m+1}}}{\mathrm{arg}\min} \quad \frac{1}{2} \mathbf{u}^T H(x) \mathbf{u} + F^T(x) \mathbf{u} \tag{QP}  \\
			& \mathrm{s}.\mathrm{t}. \nonumber \\
			\tag{CLF}
			& L_f V(x) + L_g V (x) u \leq - \alpha_v(|x|) + \delta \nonumber \\
			\tag{CBF}
			& L_f h(x)  + L_g h(x) u \geq - \alpha(h(x)),
\end{align}
\vspace{-2mm}
\end{small}

\noindent where here $H(x) \in \R^{(m+1) \times (m+1)}$ and $F(x) \in R^{m+1}$ are arbitrary smooth functions that can be chosen based on the type of control inputs. $\delta>0$ is the relaxation term used to ensure feasibility of the QP. It can be verified that this type of control law $u^*(x)$, with $H$ positive definite, is Lipschitz continuous and renders the set $\mathcal{C}$, defined by $h$, forward invariant \cite[Theorem 2]{7040372}. It can also be verified that \eqref{eq:CLFCBFQP} may not necessarily guarantee forward invariance of $\mathcal{C}$ under input disturbances\footnote{Due to uncertainties in the system, an unknown disturbance term $d(t)$ gets added to the control law obtained from \eqref{eq:CLFCBFQP}.} 
(see Example \ref{ex:arctan}). We will therefore utilize the following QP formulation:

\begin{small}
\begin{align}
\label{eq:ISSfCLFCBFQP}
u^*(x) = & \underset{{\mathbf{u} = (u,\delta)\in \R^{m+1}}}{\mathrm{arg}\min} \quad \frac{1}{2} \mathbf{u}^T H(x) \mathbf{u} + F^T(x) \mathbf{u} \tag{ISSf-QP}  \\
& \mathrm{s}.\mathrm{t}. \nonumber \\
& L_f V(x) + L_g V (x) u \leq - \alpha_v(|x|) + \delta  \nonumber \\
\label{eq:ISSfCBF}
\tag{ISSf-CBF}
& L_f h(x)  + L_g h(x) u - \epsilon L_g h(x) L_g h(x)^T \geq - \alpha(h(x)),
\end{align}
\end{small}

\begin{figure}[t!]\centering
	\includegraphics[width=0.46\columnwidth]{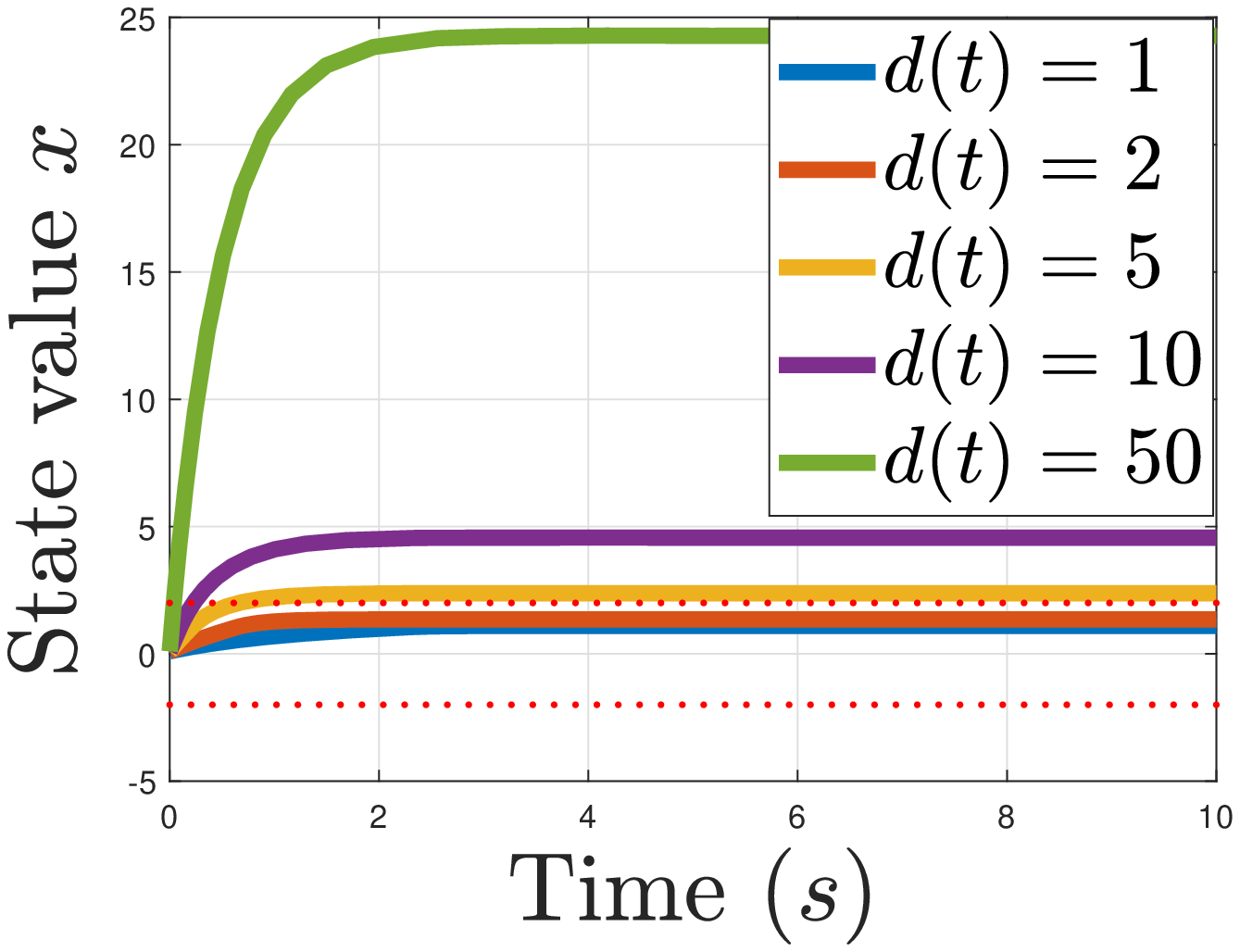}
	\includegraphics[width=0.46\columnwidth]{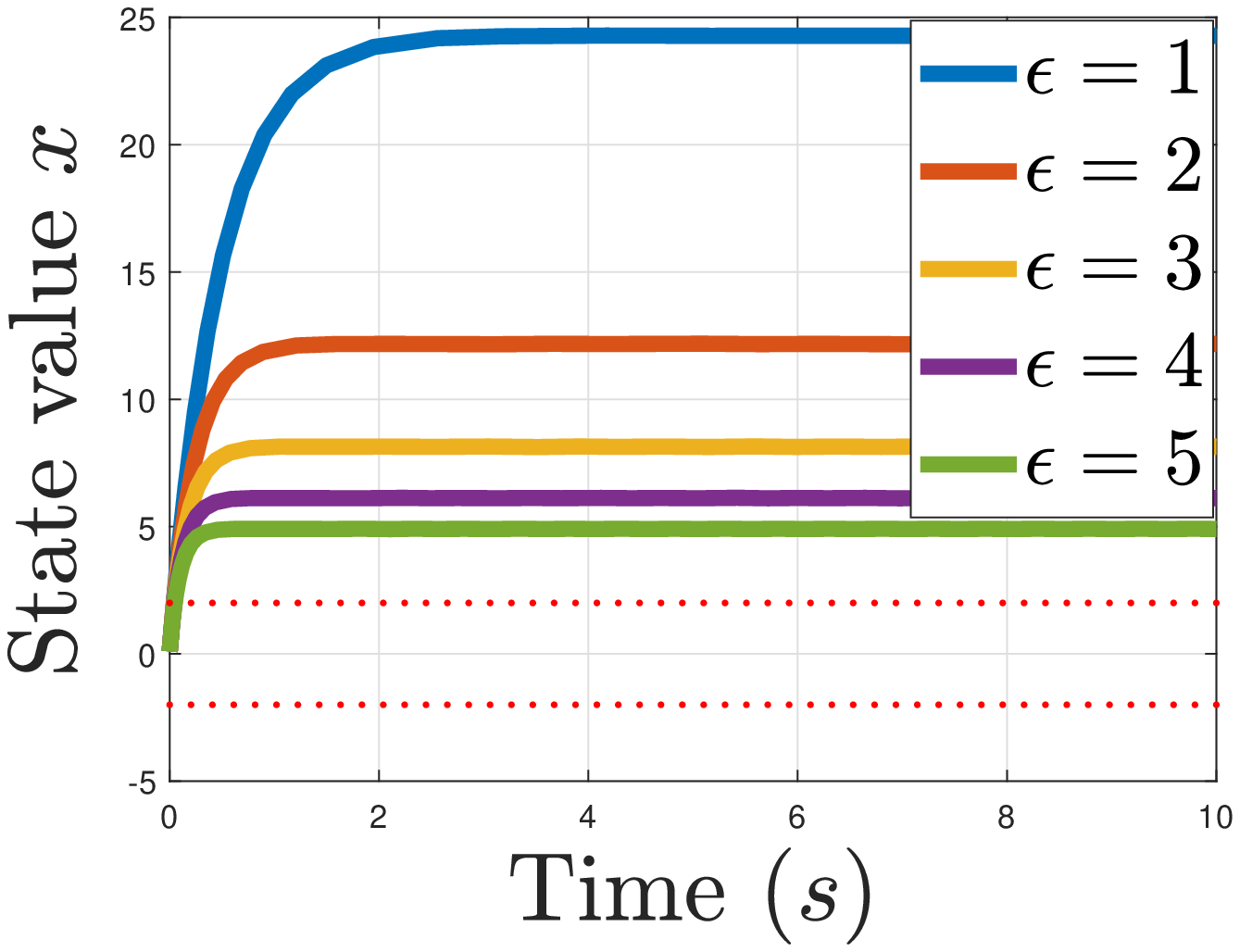}
	\caption{Figure showing the response for different values of input disturbances (left) and $\epsilon$ (right) for the controller of the type \eqref{eq:ISSfCLFCBFQP}. $H$ was chosen to be a constant diagonal matrix, and $F=0$. The dashed lines correspond to $x=\pm 2$, the boundary of the safe set. The plots show responses to constant values of $d$'s, and the boundedness is true for all bounded functions of time $d(t)$.}
	\label{fig:arctanexample}
\end{figure}

\noindent which will ensure forward invariance of a slightly larger set $\mathcal{C}_d$. Note the inclusion of a new user defined $\epsilon > 0$ in \eqref{eq:ISSfCBF}, which indirectly helps in restricting $\mathcal{C}_d$ to a smaller region. This will be more clear from the examples below.


\begin{example}\label{ex:arctan}
Consider the system
\begin{align}
 \dot x = - \tan^{-1}(x) + u,
\end{align}
with the following CLF and CBF candidates:
\begin{align}
\label{eq:clfcbfarctan}
&{\rm{CLF}}: \quad V(x) = x \tan^{-1}(x) \nonumber \\
&{\rm{CBF}}: \quad h(x) = 4 - x^2. 
\end{align}
It can be verified that $V$ is a valid Lyapunov function for $u \equiv 0$ \cite[pp. 4]{angeli2000characterization}, hence a valid CLF. 
The CBF $h$ ensures that the state $x$ stays in the interval $[-2,2]$. It can also be verified that $h$ is a valid CBF:
\begin{align}
\resizebox{1\hsize}{!}{$
 \dot h(x,u) = 2 x \tan^{-1} x -2 x u \geq -2(2\tan^{-1} 2 - x \tan^{-1} x) -2 x u .\nonumber
 $}
 \end{align}
 Here $u$ can be suitably picked in such a way that $\dot h(x,u) \geq - \alpha(h(x))$,
where $\alpha(h(x)):=2(2\tan^{-1} 2 - x \tan^{-1} x)$, which is a valid extended class $\classK$ function w.r.t. $h$. A controller of the form \eqref{eq:CLFCBFQP} will not guarantee ISSf of $\mathcal{C} = \{ x\in \R :  4 - x^2 \geq 0\}$ (take $x(0)=0.1$, $u=0$ and $d(t)=10$). On the other hand, the controller of the form \eqref{eq:ISSfCLFCBFQP}, indeed, yields ISSf of $\mathcal{C}$. \figref{fig:arctanexample} shows comparisons for the controller of the type \eqref{eq:ISSfCLFCBFQP} with different values of $d$ and $\epsilon$.
\end{example}
\begin{example}\label{ex:robot}
Consider a 2-DOF robot example given by \figref{fig:robotarm}. We have the following dynamics:
\begin{align}\label{eq:dynamicsrobot}
    \begin{bmatrix}
      m r^2  + \frac{M L^2}{3} & 0 \\
      0 & m    
    \end{bmatrix} \begin{bmatrix} 
		    \ddot \theta \\ 
		    \ddot r 
		  \end{bmatrix} + \begin{bmatrix}  2 m r \dot r \dot \theta \\ - m r \dot \theta^2 \end{bmatrix} = u =  \begin{bmatrix} \tau \\ T \end{bmatrix},
\end{align}
where $m=1$ $kg$ is the mass of the second link, $M=1 \: kg$ is the mass of the longer link with length $L=3\:m$, $\theta\:rad$ is the rotation, and $r$ is the linear displacement along the axis. $\tau$ is the torque and $T$ is the force acting on the corresponding joints. The linear displacement $r$ has a limit $r^*=2\:m$.

The goal is to drive the configuration $q = (\theta, r)$ to some constant desired values $q_d = (\theta_d, r_d)$. CLF and CBF candidates are
\begin{align}
\label{eq:robotclfcbf}
&{\rm{CLF}}: \quad V(x) = (q - q_d)^T K_p (q - q_d) + \dq^T D(\q) \dq \nonumber \\
&{\rm{CBF}}: \quad h(x) = r^* - r,  
\end{align}
where $x=(q,\dq)$, $D$ is the inertia matrix obtained from \eqref{eq:dynamicsrobot}, and $K_p$ is a diagonal gain matrix (chosen to be identity). 
In order to verify that $V$ is a valid CLF, we need to determine a control law $u$ that results in asymptotic convergence of $V(x(t))$. In fact, PD control laws are sufficient for asymptotic convergence of $V$ \cite{arimoto1984stability}. Therefore, we have the following CLF based semi-definite constraint along with the relaxation $\delta$:
\begin{align}
	L_f V (x) + L_g V(x) u \leq - \dq^T K_d \dq + \delta,
\end{align}
with the gain $K_d$ chosen to be identity.
\begin{figure}[t!]\centering
	\includegraphics[height=0.4\columnwidth]{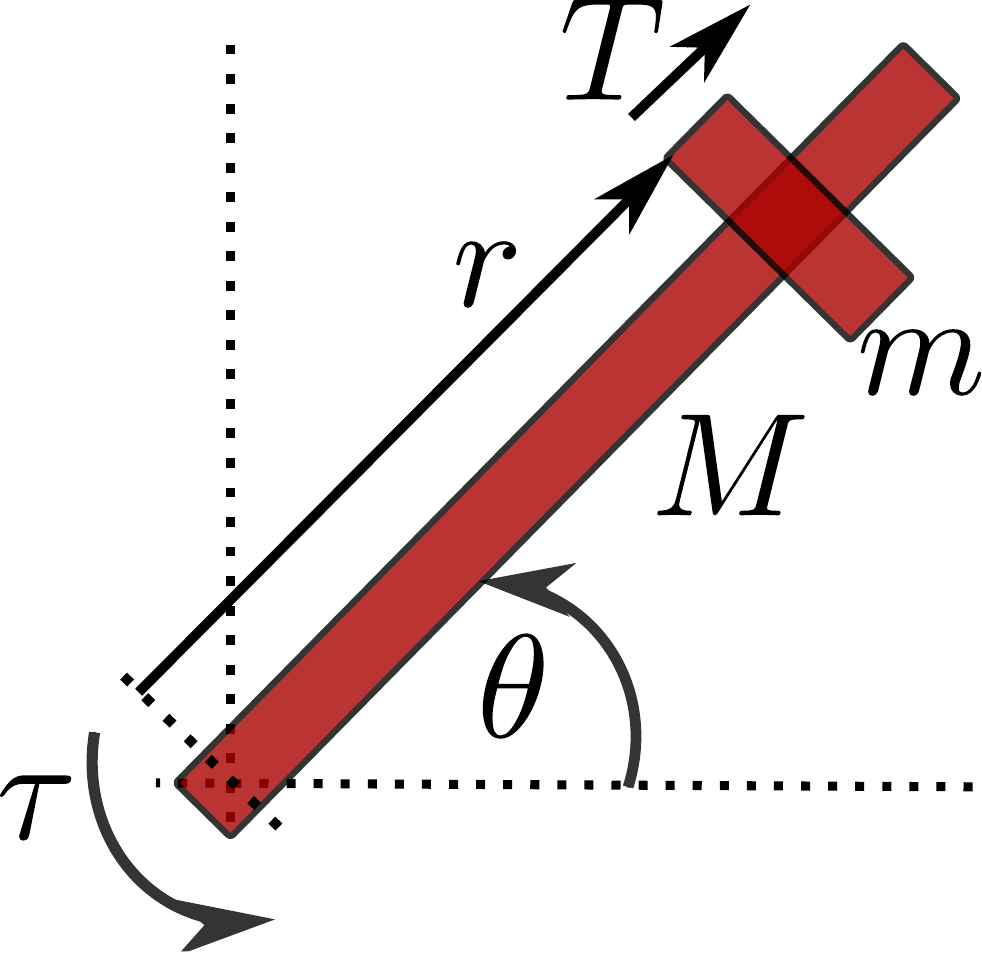}
	\includegraphics[height=0.4\columnwidth]{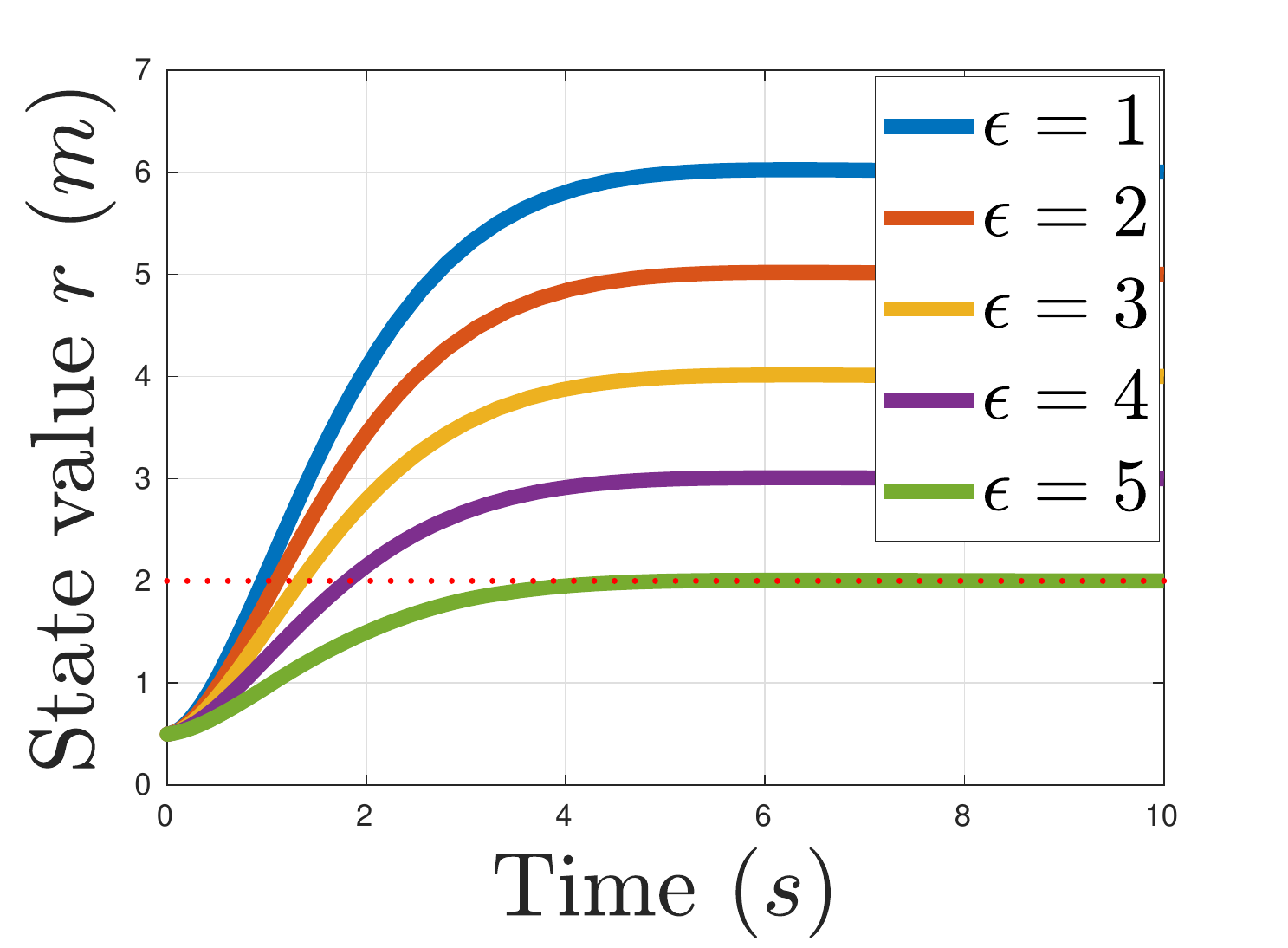}
	\caption{Left figure is showing a 2-DOF robotic system. Gravity is not included in the model. Right figure is showing the plots of $r(t)$ w.r.t. time for different values of $\epsilon$.}
	\label{fig:robotarm}
\end{figure}

Having obtained the CLF based constraint, we now determine a suitable ISSf-CBF based constraint.
The goal is to ensure ISSf of $\mathcal{C} = \{ (\theta,r,\dot \theta,\dot r) : r^* - r \geq 0\}$ under bounded disturbances.
Since $h$ (defined by \eqref{eq:robotclfcbf}) is relative degree two, we will choose {\it exponential} type of barrier functions \cite{7524935}. We have the following ISSf-CBF based constraint
\begin{align}
 \label{eq:ebf}
 & \mu_b = - [k_p  \:\: k_d ]\eta_b , \quad \eta_b =  \begin{bmatrix} h \\ L_f h \end{bmatrix} \nonumber \\
 & L_f^2 h(x) + L_g L_f h(x) u   - \epsilon L_g L_f h(x) L_g L_f h(x)^T \geq \mu_b,
\end{align}
for some positive constants $k_p,k_d, \epsilon$. We have a new ISSf-QP formulation with this new constraint instead of \eqref{eq:ISSfCBF}. \figref{fig:robotarm} shows the response of $r$ for this new controller as a function of time. Responses are shown for different values of $\epsilon$ subject to disturbance $d=5$. We chose $k_p=1$, $k_d=1.7321$. Note that $\mathcal{C}_d$ shrinks for larger values of $\epsilon$. 
%
\end{example}



\section{Conclusions}
In this letter we formally defined the notion of {\it input-to-state safety} w.r.t. sets, and the associated {\it input-to-state safe control barrier functions} that ensure forward invariance of sets under input disturbances. We have also presented methods to construct ISSf-CBFs from the existing CBF formulation. Theorem \ref{lm:zcbflemma1} is exactly in the same flavor of \cite{sontag1989smooth}, wherein, for affine control systems, input-to-state stabilizing controllers were constructed via Lyapunov functions i.e., $k(x) - L_g V(x)^T$. 
Future work will involve a detailed analysis of the different properties of ISSf-CBFs. 
The hope is that the formulations presented will lay the groundwork for safety-critical control that is robust to disturbances.

\bibliographystyle{IEEEtran}
\bibliography{bibdata}

\begin{thebibliography}{10}
\providecommand{\url}[1]{#1}
\csname url@samestyle\endcsname
\providecommand{\newblock}{\relax}
\providecommand{\bibinfo}[2]{#2}
\providecommand{\BIBentrySTDinterwordspacing}{\spaceskip=0pt\relax}
\providecommand{\BIBentryALTinterwordstretchfactor}{4}
\providecommand{\BIBentryALTinterwordspacing}{\spaceskip=\fontdimen2\font plus
\BIBentryALTinterwordstretchfactor\fontdimen3\font minus
  \fontdimen4\font\relax}
\providecommand{\BIBforeignlanguage}[2]{{%
\expandafter\ifx\csname l@#1\endcsname\relax
\typeout{** WARNING: IEEEtran.bst: No hyphenation pattern has been}%
\typeout{** loaded for the language `#1'. Using the pattern for}%
\typeout{** the default language instead.}%
\else
\language=\csname l@#1\endcsname
\fi
#2}}
\providecommand{\BIBdecl}{\relax}
\BIBdecl

\bibitem{7040372}
A.~D. Ames, J.~W. Grizzle, and P.~Tabuada, ``Control barrier function based
  quadratic programs with application to adaptive cruise control,'' in
  \emph{53rd IEEE Conference on Decision and Control}, 12 2014, pp. 6271--6278.

\bibitem{7782377}
A.~D. Ames, X.~Xu, J.~W. Grizzle, and P.~Tabuada, ``Control barrier function
  based quadratic programs for safety critical systems,'' \emph{IEEE
  Transactions on Automatic Control}, vol.~62, no.~8, pp. 3861--3876, 8 2017.

\bibitem{7524935}
Q.~Nguyen and K.~Sreenath, ``Exponential control barrier functions for
  enforcing high relative-degree safety-critical constraints,'' in \emph{2016
  American Control Conference (ACC)}, 7 2016, pp. 322--328.

\bibitem{7039737}
M.~Z. Romdlony and B.~Jayawardhana, ``Uniting control {L}yapunov and control
  barrier functions,'' in \emph{53rd IEEE Conference on Decision and Control},
  12 2014, pp. 2293--2298.

\bibitem{XU201554}
\BIBentryALTinterwordspacing
X.~Xu, P.~Tabuada, J.~W. Grizzle, and A.~D. Ames, ``Robustness of control
  barrier functions for safety critical control,'' \emph{IFAC-PapersOnLine},
  vol.~48, no.~27, pp. 54 -- 61, 2015, analysis and Design of Hybrid Systems
  ADHS. [Online]. Available:
  \url{http://www.sciencedirect.com/science/article/pii/S2405896315024106}
\BIBentrySTDinterwordspacing

\bibitem{PRAJNA2005526}
\BIBentryALTinterwordspacing
S.~Prajna and A.~Rantzer, ``On the necessity of barrier certificates,''
  \emph{IFAC Proceedings Volumes}, vol.~38, no.~1, pp. 526 -- 531, 2005, 16th
  IFAC World Congress. [Online]. Available:
  \url{http://www.sciencedirect.com/science/article/pii/S1474667016367556}
\BIBentrySTDinterwordspacing

\bibitem{WIELAND2007462}
\BIBentryALTinterwordspacing
P.~Wieland and F.~Allg{\"o}wer, ``Constructive safety using control barrier
  functions,'' \emph{IFAC Proceedings Volumes}, vol.~40, no.~12, pp. 462 --
  467, 2007, 7th IFAC Symposium on Nonlinear Control Systems. [Online].
  Available:
  \url{http://www.sciencedirect.com/science/article/pii/S1474667016355690}
\BIBentrySTDinterwordspacing

\bibitem{xu2017realizing}
\BIBentryALTinterwordspacing
X.~Xu, T.~Waters, D.~Pickem, P.~Glotfelter, M.~Egerstedt, P.~Tabuada, J.~W.
  Grizzle, and A.~D. Ames, ``Realizing simultaneous lane keeping and adaptive
  speed regulation on accessible mobile robot testbeds,'' in \emph{Control
  Technology and Applications (CCTA), 2017 IEEE Conference on}.\hskip 1em plus
  0.5em minus 0.4em\relax IEEE, 2017, pp. 1769--1775. [Online]. Available:
  \url{http://ames.caltech.edu/xu2017realizing.pdf}
\BIBentrySTDinterwordspacing

\bibitem{wang2017safe}
\BIBentryALTinterwordspacing
L.~Wang, A.~D. Ames, and M.~Egerstedt, ``Safe certificate-based maneuvers for
  teams of quadrotors using differential flatness,'' in \emph{Robotics and
  Automation (ICRA), 2017 IEEE International Conference on}.\hskip 1em plus
  0.5em minus 0.4em\relax IEEE, 2017, pp. 3293--3298. [Online]. Available:
  \url{http://ames.caltech.edu/wang2017safe.pdf}
\BIBentrySTDinterwordspacing

\bibitem{wang2017safety}
\BIBentryALTinterwordspacing
------, ``Safety barrier certificates for collisions-free multirobot systems,''
  \emph{IEEE Transactions on Robotics}, vol.~33, no.~3, pp. 661--674, 2017.
  [Online]. Available: \url{http://ames.caltech.edu/wang2017safety.pdf}
\BIBentrySTDinterwordspacing

\bibitem{RSS2017_DiscreteTerrain_Walking}
Q.~Nguyen, X.~Da, W.~Martin, H.~Geyer, J.~W. Grizzle, and K.~Sreenath,
  ``Dynamic walking on randomly-varying discrete terrain with one-step
  preview,'' in \emph{Robotics: Science and Systems (RSS)}, 2017.

\bibitem{sontag1995characterizations}
E.~D. Sontag and Y.~Wang, ``On characterizations of the input-to-state
  stability property,'' \emph{Systems \& Control Letters}, vol.~24, no.~5, pp.
  351--359, 1995.

\bibitem{doi:10.1137/050645178}
\BIBentryALTinterwordspacing
S.~Prajna and A.~Rantzer, ``Convex programs for temporal verification of
  nonlinear dynamical systems,'' \emph{SIAM Journal on Control and
  Optimization}, vol.~46, no.~3, pp. 999--1021, 2007. [Online]. Available:
  \url{https://doi.org/10.1137/050645178}
\BIBentrySTDinterwordspacing

\bibitem{7799254}
M.~Z. Romdlony and B.~Jayawardhana, ``On the new notion of input-to-state
  safety,'' in \emph{2016 IEEE 55th Conference on Decision and Control (CDC)},
  12 2016, pp. 6403--6409.

\bibitem{romdlony2017robustness}
------, ``Robustness analysis of systems' safety through a new notion of
  input-to-state safety,'' \emph{arXiv preprint arXiv:1702.01794}, 2017.

\bibitem{sontag1989smooth}
E.~D. Sontag, ``Smooth stabilization implies coprime factorization,''
  \emph{IEEE Transactions on Automatic Control}, vol.~34, no.~4, pp. 435--443,
  4 1989.

\bibitem{sontag1989universal}
------, ``A {`}universal{'} construction of artstein's theorem on nonlinear
  stabilization,'' \emph{Systems \& control letters}, vol.~13, no.~2, pp.
  117--123, 1989.

\bibitem{angeli2000characterization}
D.~Angeli, E.~D. Sontag, and Y.~Wang, ``A characterization of integral
  input-to-state stability,'' \emph{IEEE Transactions on Automatic Control},
  vol.~45, no.~6, pp. 1082--1097, 2000.

\bibitem{arimoto1984stability}
S.~Arimoto, ``Stability and robustness of {PID} feedback control for robot
  manipulators of sensory capability,'' \emph{International Journal of Robotics
  Research}, pp. 783--799, 1984.

\end{thebibliography}

\end{document}